\documentclass[preprint]{elsarticle}
\usepackage{graphicx, amssymb, mathscinet, amsmath, amsthm, dsfont}
\usepackage{xcolor}
\newtheorem{theorem}{Theorem}

\newtheorem{definition}[theorem]{Definition}
\newtheorem{corollary}[theorem]{Corollary}
\newtheorem{lemma}[theorem]{Lemma}

\newtheorem{remark}{Remark}

\def\qed{\hbox{${\vcenter{\vbox{		 %HOLLOW SQUARE
   \hrule height 0.4pt\hbox{\vrule width 0.4pt height 6pt
   \kern5pt\vrule width 0.4pt}\hrule height 0.4pt}}}$}}

\def\cF{\mathcal F}

\def\bC{\mathbb C}

\def\bE{\mathbb E}

\def\bN{\mathbb N}

\def\bP{\mathbb P}

\def\bR{\mathbb R}

\def\uh{\underline{h}}
\def\uu{\underline{u}}
\def\uz{\underline{z}}

\def\uu{\underline u}

\begin{document}

\title{An It\^o calculus for a class of limit processes arising from random walks on the complex plane}

\author[tn]{Stefano Bonaccorsi}
\ead{stefano.bonaccorsi@unitn.it}

\author[msu]{Craig Calcaterra}
\ead{craig.calcaterra@metrostate.edu}

\author[tn]{Sonia Mazzucchi}
\ead{sonia.mazzucchi@unitn.it}

\address[tn]{Dipartimento di Matematica, Universit\`a di
    Trento, via Sommarive 14, 38123 Povo (Trento), Italia
}

\address[msu]{Applied Mathematics Department, 
Metropolitan State University, 700 E. Seventh St. 
St. Paul MN  55106-5000, USA
}
%\thanks{Submitted Xxxx xx, xxxx. Published Xxxxx xx, xxxx.}
% \thanks{The research of G. Ziglio was supported by Project NeSt funded
%   by Provincia Autonoma di Trento (P.A.T.) within \emph{Bando unit\`a
%     di ricerca 2006}}


\begin{abstract}
Within the framework of the previous paper \cite{Bonaccorsi2015}, we develop a generalized stochastic calculus for processes associated to higher order diffusion operators.
Applications to the study of a Cauchy problem, a Feynman-Kac formula and a representation formula for higher derivatives of analytic functions are also given.
%A probabilistic construction for the solution of a general class of high order heat-type equations is constructed in terms of the scaling limit of  random walks in the complex plane. 
\end{abstract}

\maketitle

\def\uh{\underline{h}}
\def\uu{\underline{u}}
\def\uz{\underline{z}}

%==========================================================

\section{Introduction}
One of the main instances of the fruitful interplay between analysis and probability is the connection between parabolic equations associated to second-order elliptic operators and the theory of Markov processes. 
The main consequence of this extensively studied topic is  the famous  {\em Feynman-Kac formula}, providing a representation of the solution of the heat equation with potential $V\in C_0^\infty (\bR^d)$ (the continuous functions vanishing at infinity)
\begin{equation} \label{heat}\left\{ \begin{array}{l}
\frac{\partial}{\partial t}u(t,x)=\frac{1}{2}\Delta u(t,x) -V(x)u(t,x),\qquad t\in \bR^+, x\in\bR^d\\
u (0,x)=f(x)\\
\end{array}\right. \end{equation}
in terms of an integral with respect to the distribution of the Wiener process, the mathematical model of the Brownian motion:
\begin{equation}\label{Fey-Kac0}
u (t,x)=\int_{C_t} e^{-\int_0^tV(\omega(s)+x)ds}f( \omega(t)+x) \, {\rm d}W(\omega).\end{equation}
In fact a probabilistic representation of this form cannot be written in the case of semigroups whose generator does not satisfy the maximum principle. In particular
if the Laplacian in Eq \eqref{heat} is replaced with an higher order differential operator, i.e. if we consider a Cauchy problem of the form
\begin{equation} \label{heatN}\left\{ \begin{array}{l}
\frac{\partial}{\partial t}u(t,x)=(-1)^{n+1} \Delta^n u(t,x) -V(x)u(t,x),\quad t\in \bR^+, x\in\bR^d, \\
u (0,x)=f(x),\\
\end{array}\right. \end{equation}
with $n\in\bN$, $n\geq 2$, 
then a formula analogous to \eqref{Fey-Kac0},  giving the solution of \eqref{heatN} in terms of the expectation with respect to the measure associated to a Markov process, is lacking and it is not possible to find a stochastic process  which plays for the parabolic equation \eqref{heatN} the same role that the Wiener process plays for the heat equation. 
The problem of how to overcome this limitation  has been studied by means of different techniques and two main approaches have been proposed. The first one was introduced by V. Yu. Krylov in 1960 \cite{Kry} and further developed by K. Hochberg in 1978 \cite{Hoc78}. The  solution of \eqref{heatN} is constructed in terms of the expectation with respect to a {\em signed} measure with infinite total variation on a space of  paths on the interval $[0,t]$.  This approach is related to the  theory of {\em pseudoprocesses}, i.e. processes associated to signed instead of probability measures. 
It is important to recall that  due to the particular conditions necessary for the  generalization of the Kolmogorov existence theorem for the  limit of a projective system of complex measures (see \cite{Thomas}), in the case of Krylov-Hochberg process, a well defined signed measure on $\bR^{[0,t]}$ cannot exist and the "integrals" realizing the Feynman-Kac formula for equation \eqref{heatN} are just formal expressions which cannot make sense in the framework of Lebesgue integration theory but are to be meant as limit of a particular approximating sequence. However, even taking into account these technical problems, an analog of the arc-sine law \cite{Hoc78,HocOr94,Lachal}, of the  central limit theorem \cite{Hoc79,SmoFa} and of It\^o formula and Ito stochastic calculus \cite{Hoc78,Ni} have been developed for the (finite additive)  Krylov-Hochberg signed measure. 
For a extensive discussion of these problems in the framework of a generalized integration theory on infinite dimensional spaces as well as for a unified view of probabilistic and complex integration see \cite{AlMa14,Albeverio2015}. 
It is worthwhile to mention  the work by D. Levin and T. Lyons \cite{LevLyo} on rough paths, conjecturing that the above mentioned  signed measure could be finite if defined on the quotient space of equivalence classes of paths corresponding to different parametrization of the same path.

A different approach, introduced by T. Funaki \cite{Fu} for the case where $n=2$,   is based on the construction of a stochastic process (with dependent increments) on the complex plane. Funaki's process is obtained by composing two independent Brownian motions and has some relations with the {\em iterated Brownian motion} \cite{Bur,Allouba2002}. Furthermore this approach is related with the theory of Bochner subordination \cite{Boch} and can be applied to  partial differential equations of even and odd order $2^n$, by multiple iterations of suitable processes \cite{Fu,HocOr96,OrZha,OrDOv}.
Complex valued processes, connected to PDE of the form \eqref{heatN} have been also proposed by other authors by means of different techniques \cite{Leandre}.
In \cite{MadRyb93,BurMan96} K. Burdzy and A. Madrecki consider equation \eqref{heat-N} with $n=2$  and $V\equiv 0$,  constructig a  probabilistic representation for its solution in terms of a  stable probabilistic Borel measure $m$ on the space $\Omega=C([0,t],\bC^\infty)$ of continuous mappings on $[0,t]$ with values in the set $\bC^\infty$ of complex valued sequences, endowed with the product topology. In this setting a Feynman-Kac type formula is  proved for the fourth order heat-type equation with   linear potential.
$ \frac{\partial u}{\partial t}=\frac{1}{8}\frac{\partial^{4} u}{\partial x^{4}}+(iax+b)u.$
By means of the theory of infinite dimensional Fresnel integral, in \cite{Ma14} a Feynman-Kac-type formula for equation \eqref{heatN} has been proved for  potentials $V$ which are Fourier transform of complex bounded measures on the real line.

In a recent paper \cite{Bonaccorsi2015}, two of the authors have introduced a different approach which is rather simple and elegant, providing the solution  of the parabolic Cauchy problem
\begin{equation}\label{e:in1}
\begin{cases}
\partial_t u(t,x) = \frac{\alpha}{N!} \, \partial^N_x u(t,x),
\\
u(0,x) = f(x), \qquad x \in \bR,
\end{cases}
\end{equation}
(where $N > 2$ is an integer constant, $\alpha \in \bC$  and $f: \bR \to \bC$ is the initial datum) in terms of the expectation with respect to a the law of a sequence of random walks $\{W_n^N,\ n \in \bN\}$ on the complex plane. The main idea in \cite{Bonaccorsi2015} is the construction of a sequence of processes  which play for the PDE \eqref{e:in1} of order $N$ the same role that the Wiener process plays for the heat equation. 
Let $(\Omega, \cF, \bP)$ be a probability space,  $\alpha \in \bC \setminus \{0\}$ be a complex number and $N \geq 1$ a given integer.
Let 
\begin{align*}
R(N)= \{e^{2 i \pi k/N},\ k=0,1,\dots, N-1\}
\end{align*} 
be the roots of unity.
Throughout the paper $\xi$ will denote a uniformly distributed random variable on the set $\alpha^{1/N} R(N)$.
Given a sequence $\{\xi_k\}$ of i.i.d.\ random variables $\xi_k \sim \xi$, we define the complex stochastic process
\begin{align}\label{defWn}
W^n(t) = \frac{1}{n^{1/N}} \sum_{k=1}^{\lfloor nt \rfloor} \xi_k.
\end{align}
For each $n$, the process $W^n$ is a random walk on the complex plane, which is geometrically aligned and precisely scaled to extract information about the $N$-th complex derivative of analytic functions \cite{Bonaccorsi2015}. 
This claim will be made clear in the course of the paper.

If $\alpha =1$ and $N=2$ the process $W_n$ converges weakly  \cite{Bil} to the Wiener process and the Feynman-Kac formula \eqref{Fey-Kac0} (for $V\equiv 0$) can be written as
\begin{equation}\label{Fey-Kac0bis}
u(t,x)=\lim_{n\to \infty} \bE\left[ f(W_n(t)+x)\right].
\end{equation}
For $N>2$ the sequence of  processes $W_n$ cannot  converge  because of the particular scaling exponent $1/N$ (directly related with the order of the PDE \eqref{e:in1}) appearing in the denominator on the right hand side of \eqref{defWn}. However for a restricted class of functions $f:\bC\to \bC$ the limit in \eqref{Fey-Kac0bis} still exists providing a representation for the solution of \eqref{e:in1} (see \cite{Bonaccorsi2015}).

Notice that $W^n$ are in fact pre-Brownian motions, rescaled so they diverge in the limit due to the factor $n^{-1/N}$ when $N > 2$.
Asymptotically their paths have the geometric properties of Brownian motion. To see this, we calculate the real $2$-dimensional variance
\begin{align*}
\bE \left( |W^n(t) - \bE W^n(t)|^2 \right) = \bE|W^n(t)|^2 = \bE \left( \left|n^{-1/N} \sum_{k=1}^{\lfloor n t \rfloor} \xi_k \right|^2 \right)
%\\
= n^{-2/N} \left(\sum_{k=1}^{\lfloor n t \rfloor} \bE |\xi_k|^2 \right) = \frac{{\lfloor n t \rfloor}}{n^{2/N}}
\end{align*}
Thus if we reparametrize to the process $\bar W^n(t) = W^n(t/n^{1-2/N})$ the random walk $\bar W^n(t)$ has 
real 2-dimensional variance $\frac{{\lfloor n^{2/N} t \rfloor}}{n^{2/N}}$ and so, by Donsker's Theorem, 
converges to a 2-dimensional Brownian motion
as $n \to \infty$ for $N > 2$.
In this sense (since paths do not depend on the parametrisation of the curves) the paths of $W^n$ 
have the geometric properties of paths of Brownian motion in the limit as $n \to \infty$.

\medskip

%When $N = 2$ the processes $W^n$ converge to the 1-dimensional Wiener process $W$. 
%In the case $N > 2$ the random walks $W^n$ stand in for $W$ in a faithful generalization of It\^o calculus 
%which emphasizes the $N$-th derivative, where It\^o calculus emphasizes the 2-nd. 
%Heuristically, we imagine the samples $W^n(\cdot,\omega)$ as lightning bolts -- infinitely fast, 
%fractally nonsmooth curves probing analytic functions at infinity to analyze the $N$-th derivative.

In this paper, we improve the construction of \cite{Bonaccorsi2015} with the development of an It\^o's calculus for 
the limit law of the random walks $W^n$. To achieve our aim,
we first analyze the behaviour of the random walks $W^n$, with particular regard to the estimates on the moments and Fourier's transform of $W^n$.
\\
Then, we introduce the analog of It\^o's calculus. Heuristically, the rescaling $n^{1/N}$ in the construction of $W^n$ (by comparing with the classical construction, that
correspond to the case $N=2$)
implies that only the $N$-th moment behaves like ${\rm d}t$ in the limit, thus
implying that all moments of lesser order shall be considered in the It\^o formula,
while we can neglect moments of higher order in the limit.
\\
Moreover, integrals with respect to the increments of the random walk $W^n$ and their moments lesser than $N$
shall be analogous to {\em stochastic} integrals and, consequently, shall have zero mean.
In Theorem \ref{t:main-claim} we prove that this intuition is indeed true, and we prove 
the following It\^o formula (compare \eqref{eq:unify1})
\begin{align*}
\lim_{n \to \infty} \bE[f(z + W^n(t))] - f(z) = \frac{\alpha}{N!}  \int_0^t \lim_{n \to \infty} \bE[\partial^{N} f(z + W^n(s))] \, {\rm d}s.
\end{align*}
By setting
\begin{align*}
u(t,z) = \lim_{n \to \infty} \bE[f(z + W^n(t))],
\end{align*}
this is equivalent to saying $u(t,z)$ is a classical solution of the $N$-th order Cauchy problem \eqref{e:in1}.
%\begin{align*}
%\frac{\partial}{\partial t}u(t,z) &= \frac{\alpha}{N!}\partial_x^{N} u(t,z)
%\\
%u(0,z) &= g(z).
%\end{align*}

We shall then extend this result in two directions. First, through a rather straightforward extension of previous computations,
we allow for a time-dependent coefficient in front of the diffusion operator, constructing the probabilistic representation for the solution of \begin{equation}
\begin{cases}
\partial_t u(t,x) = \frac{\alpha}{N!} a(t) \partial^N_x u(t,x),
\\
u(0,x) = f(x), \qquad x \in \bR.
\end{cases}
\end{equation}
Then we begin the analysis of a Feynman-Kac formula, and we show that for a time dependent potential which is linear in the space variable, the classical solution of the initial value problem
\begin{equation*}
\begin{cases}
\frac{\partial}{\partial t}u(t,z) = \frac{\alpha}{N!}\partial_x^{N} u(t,z)+V(t,x)u(t,x),
\\
u(0,x) = f(x), \qquad x \in \bR.
\end{cases}
\end{equation*}
is given by 
$$u(t,x)=\lim_{n\to\infty}\bE\left[ f(x+W^n(t))e^{\int_0^t V(t-s, x+W^n(s))ds}\right]$$
All these results require to choose a suitable class of functions for $f$. The first, obvious remark,
is $f$ must be extensible to the complex plane. Moreover, it is necessary to have estimates on the function and all its derivatives in $\bC$.
For the sake of simplicity, we shall limit ourselves to the classical case of analytic functions of exponential type, see for instance \cite{Carmichael1934} and definition \ref{def-exp}.\\
Finally we study the properties of stopping time for the process $W^n$, proving a suggestive formula for the $N$-th order derivative of an analytic function %at $z=0$: 
%$$f^N(0)=\frac{N!}{\alpha}\lim_{n\to \infty}\bE\left[\frac{1}{\tau_n}\left(f(W^n(\tau_n))-f(0)\right)\right]$$
$$f^N(z)=\frac{N!}{\alpha}\lim_{n\to \infty}\bE\left[\frac{1}{\tau_n}\left(f(z+W^n(\tau_n))-f(z)\right)\right]$$
where $\tau_n$ is the exit time of $W^n$ from the ball $B(0,R)\subset \bC$.

%{\color{red}{\bf Claim.} Assume that $g(z)$ is an analytic function of exponential type $c$. Then, for every $z \in \bC$,
%the function $f(x) = g(z+x)$ is an analytic function of exponential type $c$.}

%\vskip 4\baselineskip

\section{Random walk on the complex plane}\label{sez1}

%\subsection{Powers, moments and characteristic function of $W^n(t)$}
The present section is devoted to the proof of some properties of the sequence of complex random walks $W^n(t)$.

Let $\alpha \in \bC$ and $N\in \bN$ with $N>2$ and let us consider the complex random variable $\xi$ uniformly distributed on the set $\alpha R(N)$, where $R(N)=\{e^{i2\pi k/N}, k=0, 1,...,N-1\}$ is the set of $N-$roots of unity.
We have
\begin{equation}
\label{e2}
\bE[f(\xi)] = \frac{1}{N} \, \sum_{k=0}^{N-1} f(\alpha^{1/N} e^{2 i \pi k/ N}).
\end{equation}

\begin{lemma}\label{lp1}
The random variable $\xi$ has finite powers of every order
\begin{equation}
\label{e3}
\bE[\xi^m] =
\begin{cases}
\alpha^{m/N}, & m = n N,\ n \in \bN,
\\
0, & \text{otherwise}
\end{cases}
\end{equation}
and finite moments of every order
\begin{align*}
\bE[|\xi|^m] = |\alpha|^{m/N}.
\end{align*}
\end{lemma}

\begin{proof}
We compute
\begin{align*}
\bE[\xi^m] = \frac{1}{N} \sum_{k=0}^{N-1} \alpha^{m/N} e^{2 i \pi m k/N};
\end{align*}
if $m/N = n \in \bN$, since $e^{2 i \pi n} = 1$ then each term in the sum is equal to 1;
in the other case, we employ the trigonometric sum
\begin{align*}
\sum_{k=0}^{N-1}  e^{2 i \pi m k/N} = \frac{1 - e^{2 i \pi m}}{1 - e^{2 i \pi m/N}} = 0.
\end{align*}
\end{proof}

For a complex random variable $X$ we define its {\em characteristic function}\footnotemark as
$\psi_X(\lambda) := \bE[e^{i \lambda X}]$. We have
\begin{align*}
\psi_\xi(\lambda) := \bE[e^{i \lambda \xi}] = \frac{1}{N} \, \sum_{k=0}^{N-1} \exp(i\alpha^{1/N} e^{2 i \pi k/ N}).
\end{align*}
\footnotetext{
Though it is more common to use $\bE[\exp\left(i {\rm Re}\lambda X\right)]$ as the characteristic function of a complex-valued random variable due to the connection with the 2-dimensional Fourier
transform, the complication would not improve the results in this paper
}

The next lemma, which will be applied in section \ref{sez-FK},  provides an estimate of the characteristic function of $\xi$.
\begin{lemma}\label{stime-chi}
Let $|\lambda|\leq R$. Then there exists a constant $C\in\bR$ such that 
$$
\left|\bE[e^{\lambda \xi}]-e^{\frac{\alpha\lambda^N}{N!}} \right| \leq C|\alpha|^2|\lambda|^{2N}.$$
\end{lemma}

\begin{proof}
For the properties of the random variable $\xi$, setting $z_j:=e^{i2\pi j/N}$, the function $\chi:\bC\to\bC$ defined as $\chi (\lambda):= \bE[e^{\lambda \xi}]$ is an entire analytic function with the following power series expansion: 
\begin{eqnarray*}
\chi(\lambda)&=&\frac{1}{N}\sum_{j=0}^{N-1}e^{\alpha^{1/N}\lambda z_j}=\frac{1}{N}\sum_{j=0}^{N-1}\sum_k\frac{\lambda^k\alpha^{k/N}}{k!}z_j^k\\
&=&\sum_k\frac{\lambda^k\alpha^{k/N}}{k!}\sum_{j=0}^{N-1}\frac{z_j^k}{N}\\
&=&\sum_{m}\frac{\lambda^{mN}\alpha^{m}}{(mN)!}.
\end{eqnarray*}
The difference between $\chi(\lambda)$ and $e^{\frac{\alpha\lambda^N}{N!}}$ can be estimated as:
\begin{eqnarray*}
\chi(\lambda)-e^{\frac{\alpha\lambda^N}{N!}}&=&\sum_{m=0}^\infty\lambda^{mN}\alpha^{m}\left(\frac{1}{(mN)!}-\frac{1}{m!(N!)^m}\right)\\
&=&\sum_{m=2}^\infty\lambda^{mN}\alpha^{m}\left(\frac{1}{(mN)!}-\frac{1}{m!(N!)^m}\right)\\
&=&\alpha^2\lambda^{2N}\sum_{m=0}^\infty\lambda^{mN}\alpha^{m}\left(\frac{1}{((m+2)N)!}-\frac{1}{(m+2)!(N!)^{m+2}}\right)\\
&=&\alpha^2\lambda^{2N} g(\lambda)
\end{eqnarray*}
where $g:\bC\to\bC$ is the entire analytic function defined by the power series (with infinite radius of convergence)
$$g(\lambda):=\sum_{m=0}^\infty\lambda^{mN}\alpha^{m}\left(\frac{1}{((m+2)N)!}-\frac{1}{(m+2)!(N!)^{m+2}}\right)$$
The thesis follows by the continuity of $g$ and the assumption of the boundedness of $|\lambda|$, by putting 
$$C:=\sup_{|\lambda|<R}|g(\lambda)|.$$
\end{proof}

Next we proceed to analyze the random walk $W^n$ on the complex place defined by formula \eqref{defWn}.
The main issue here is the analysis of the complex moments of the random walk and their asymptotic behavior as $n \to \infty$. 
The result that we obtain in Theorem \ref{t:unif}, is necessary in order to handle the It\^o's formula introduced in the next section.
%To calculate the $k$-th power, we use the Multinomial Theorem
%\begin{align*}
%\left( \sum_{j=1}^n x_j \right)^k = \sum_{j_1 + \dots + j_n = k} \frac{k!}{j_1! \dots j_m!} x_1^{k1} \dots x_m^{j_m}.
%\end{align*}
%The  process $W_n$ has some interesting properties. The following theorem shows the particular behaviour of the complex moments.
\begin{theorem}
\label{t:unif}
Fix $k \in \bN$, $t\in \bR^+$. 
%For $k<n$ 
The $k$-moment of $W_n(t)$ satisfies
\begin{align*}
\bE[(W_n(t))^k] =
\begin{cases}
\left(\frac{\alpha t}{N!}\right)^{k/N} \frac{k!}{(k/N)!}{\mathds 1}_{[0,\lfloor nt\rfloor]}(k/N) + R(n,k),  & k = h N,\ h \in \bN, \, 
\\
0, & \text{otherwise}
\end{cases}
\end{align*}
(${\mathds 1}_{[0,\lfloor nt\rfloor]]}$ being the indicator function of the interval $[0,\lfloor nt\rfloor]$).
For $k=0$ and $k=N$ then $R(n,hN) =0$, while $k=hN$, $h\in \bN$, $h\geq 2$, then the remainder term satisfies the inequality 
\begin{equation}\label{in-momenti}
|R(n,hN)|\leq \frac{|\alpha|^ht^{h-1}(h^2+h)}{2n}+\frac{|\alpha|^h}{n}\left(\frac{0.792hN}{\log(hN+1)}\right)^{hN} .
\end{equation}
\end{theorem}

\begin{proof}
Let $W_n(t)=\frac{1}{n^{1/N}}\sum_j^{\lfloor nt\rfloor}\xi_j$ and $\psi_n$ be its characteristic function, namely:
$$\psi_n(\lambda):=\bE[e^{i\lambda W_n(t)}]$$
We have that 
$$\bE[(W_n(t))^k] =(-i)^k\frac{d^k}{d\lambda^k} \psi_n(0),$$
where $\psi_n$ is equal to
\begin{align*}
\psi_n(\lambda) = \left( \bE[\exp(\frac{1}{n^{1/N}} i \lambda \xi)] \right)^{\lfloor nt\rfloor} = \left( \psi_\xi(\frac{\lambda}{n^{1/N}} ) \right)^{\lfloor nt\rfloor},
\end{align*}
where $\psi_\xi$ is the characteristic function of $\xi$.\\
By Fa\'a di Bruno's formula
\begin{equation}\label{Faa1}
\frac{d^k}{d\lambda^k} \psi_n(\lambda)=\sum_{\pi\in \Pi}C(|\pi|, \lambda)\prod_{B\in \pi}\left(\frac{\psi_\xi^{(|B|)}(\lambda/n^{1/N})}{n^{|B|/N}}\right)
\end{equation}
where 
$\pi$ runs over the set $\Pi$ of all partitions of the set ${1,...,k}$,  $B\in \pi$ means that the variable $B$ runs through  the list of the blocks of the partition $\pi$, $|\pi|$ denotes the number of blocks of the partition $\pi$ and $|B|$ is the cardinality of a set $B$, while the function $C:\bN\times \bR\to\bC$ is equal to
\begin{align*}
C(j,\lambda) =
\begin{cases}
\frac{\lfloor nt\rfloor!}{(\lfloor nt\rfloor-j)!}\left(\psi_\xi(\lambda/n^{1/N})\right)^{\lfloor nt\rfloor-j},  & \lfloor nt\rfloor\geq j \, 
\\
0, & \text{otherwise}
\end{cases}
\end{align*} 
 Formula \eqref{Faa1} can be written in the equivalent form:
\begin{multline}\frac{d^k}{d\lambda^k} \psi_n(\lambda)=
\\\sum \frac{k!}{m_1!m_2! \cdots m_k!}\frac{\lfloor nt\rfloor!}{(\lfloor nt\rfloor-(m_1+m_2+\dots+m_k))!}\left(\psi_\xi(\lambda/n^{1/N})\right)^{\lfloor nt\rfloor-(m_1+m_2+\dots+m_k)}\prod_{j=1}^k\left(\frac{\psi_\xi^{(j)}(\lambda/n^{1/N})}{j!n^{j/N}}\right)^{m_j}\end{multline}
where the sum is over the $k-$ple of non-negative integers $(m_1,m_2,...,m_k)$ such that $m_1+2m_2+\dots +km_k=k$ and $m_1+m_2+...+m_k\leq \lfloor nt\rfloor$.
In particular we have:
\begin{equation}\label{Faa2}
\frac{d^k}{d\lambda^k} \psi_n(0)=\sum_{\pi\in \Pi}\frac{\lfloor nt\rfloor!}{(\lfloor nt\rfloor-|\pi|)!}\prod_{B\in \pi}\left(\frac{\psi_\xi^{(|B|)}(0)}{n^{|B|/N}}\right)
\end{equation}
 where the first sum runs over the partitions $\pi$ such that $|\pi|\leq \lfloor nt\rfloor $ or equivalently
\begin{equation}\label{FaadiBruno}
\frac{d^k}{d\lambda^k} \psi_n(0)=\sum \frac{k!}{m_1!m_2! \cdots m_k!}\frac{\lfloor nt\rfloor!}{(\lfloor nt\rfloor-(m_1+m_2+\dots+m_k))!}
\prod_{j=1}^k\left(\frac{\psi_\xi^{(j)}(0)}{j!n^{j/N}}\right)^{m_j}.
\end{equation}
Since $\psi_\xi^{(j)}(0)=(i)^j\bE[\xi^j]$, and $\bE[\xi^j]\neq 0$ iff $j=mN$, with $m\in \bN$, then the product $\prod_{j=1}^k\left(\frac{\psi_\xi^{(j)}(0)}{j!n^{j/N}}\right)^{m_j}$ is non-vanishing iff $m_j=0$ for $j\neq lN$ and $k=Nm_N+2Nm_{2N}+...$, i.e. if $k$  is a multiple of $N$. Analogously in the sum appearing in formula \eqref{Faa2} the only terms giving a non-vanishing contribution correspond to those partitions $\pi$ having blocks $B$ with a number of elements which is a multiple of $N$, giving, for $k=hN$:
\begin{equation}\label{Faa3}
\frac{d^{hN}}{d\lambda^{hN}} \psi_n(0)=i^{hN}\frac{\alpha^h}{n^h}\sum_{\pi\in \Pi}\frac{\lfloor nt\rfloor!}{(\lfloor nt\rfloor-|\pi|)!}
\end{equation}
where again the sum runs over the partitions $\pi$ such that $|\pi|\leq \lfloor nt\rfloor $.
Equivalently:
\begin{eqnarray*}
\frac{d^{hN}}{d\lambda^{hN}} \psi_n(0)&=&\sum \frac{(hN)!}{(m_N)!(m_{2N})! \cdots (m_{hN})!}\frac{\lfloor nt\rfloor!}{(\lfloor nt\rfloor-(m_N+m_{2N}+\dots+m_{hN}))!}\prod_{l=1}^h\left(\frac{\psi_\xi^{(lN)}(0)}{(lN)!n^{l}}\right)^{m_{lN}},\\
&=&\sum \frac{(hN)!}{(m_N)!(m_{2N})! \cdots (m_{hN})!}\frac{\lfloor nt\rfloor!}{(\lfloor nt\rfloor-(m_N+m_{2N}+\dots+m_{hN}))!}\prod_{l=1}^h\left(\frac{i^{lN}\alpha ^l}{(lN)!n^{l}}\right)^{m_{lN}},\\
&=&\frac{i^{hN}\alpha^{h}}{n^{h}}\sum \frac{(hN)!}{(m_N)!(m_{2N})! \cdots (m_{hN})!}\frac{\lfloor nt\rfloor!}{(\lfloor nt\rfloor-(m_N+m_{2N}+\dots+m_{hN}))!}\prod_{l=1}^h\frac{1}{((lN)!)^{m_{lN}}},
\end{eqnarray*}
where the sum is over the $h-$ple of non-negative integers $(m_N,m_{2N},...,m_{hN})$ such that $m_N+2m_{2N}+...+hm_{hN}=h$ and $m_N+m_{2N}+...+m_{hN}\leq \lfloor nt\rfloor$.\\
Hence, we have 
$$
\bE[(W_n(t))^{hN}]=\alpha ^{h}\sum \frac{(hN)!}{(m_N)!(N!)^{M_N}(m_{2N})!(2N)!^{m_{2N}} \cdots (m_{hN})!(hN)!^{m_{hN}}}\frac{\lfloor nt\rfloor!}{n^{h}(\lfloor nt\rfloor-(m_N+m_{2N}+\dots+m_{hN}))!}
$$
When $n\to \infty$, the leading term in the previous sum is the one corresponding to $m_N=h$ (hence $m_2N=...=m_{hN}=0)$, which is equal to 
$$\alpha^h\frac{(hN)!}{(m_N)!(N!)^{h}}\frac{\lfloor nt\rfloor!}{n^{h}(\lfloor nt\rfloor-h)!}=\alpha^ht^h\frac{(hN)!}{h!(N!)^{h}}+\alpha^h\frac{(hN)!}{h!(N!)^{h}}\left(\frac{\lfloor nt\rfloor!}{n^{h}(\lfloor nt\rfloor-h)!}-t^h\right)$$
In the case where $\lfloor nt \rfloor <h$ then this term does not appear in the sum and  we can set it equal to 0. In the case where $\lfloor nt \rfloor \geq h$, we can estimate 
the quantity inside the brackets as:
\begin{multline*}
\left|\frac{\lfloor nt\rfloor!}{n^{h}(\lfloor nt\rfloor-h)!}-t^h\right|=
\frac{1}{n^h} \left|-(nt)^h+\prod_{j=0}^{h-1}(\lfloor nt\rfloor-j)\right|
=\frac{1}{n^h}\Big|\prod_{j=0}^{h-1}\big((\lfloor nt\rfloor-j)+(\{ nt\}+j)\big)-\prod_{j=0}^{h-1}(\lfloor nt\rfloor-j)\Big|
\\
\leq \frac{1}{n^h}\sum_{j=0}^{h-1}(\{ nt\}+j)\prod_{k\neq j}nt=\frac{(nt)^{h-1}}{n^h}\sum_{j=0}^{h-1}(\{ nt\}+j)\leq \frac{(nt)^{h-1}}{n^h}\sum_{j=0}^{h-1}(1+j)=\frac{t^{h-1}(h^2+h)}{2n}
\end{multline*}
where in the second line we have used that if $a_j,b_j\in \bR$, with $a_j,b_j\geq 0$ for all $j=0,...,m$, then %(see appendix)
$$\prod_{j=0}^m(a_j+b_j)-\prod_{j=0}^m a_j\leq \sum_{j=0}^m b_j\prod_{k\neq j}(a_k+b_k).$$
Hence 
\begin{align*}
|R_1(n,h)| = \left|\alpha^h\frac{(hN)!}{(m_N)!(N!)^{h}}\frac{\lfloor nt\rfloor!}{n^{h}(\lfloor nt\rfloor-h)!}-\alpha^ht^h\frac{(hN)!}{h!(N!)^{h}}\right|
\leq \frac{|\alpha|^ht^{h-1}(h^2+h)}{2n}
\end{align*}
By using formula \eqref{Faa3}, the remaining terms in the sum (corrisponding to the $h-$ple $(m_N,m_{2N},...,m_{hN})$ with $m_N<h$) are bounded by
\begin{eqnarray*}
R_2(n,h)&=& \frac{\alpha^h}{n^h}\sum_{\pi\in \Pi}\frac{\lfloor nt\rfloor!}{(\lfloor nt\rfloor-|\pi|)!}-\alpha^h\frac{(hN)!}{(m_N)!(N!)^{h}}\frac{\lfloor nt\rfloor!}{n^{h}(\lfloor nt\rfloor-h)!}\\
&\leq & \frac{\alpha^h}{n^h}\sum_{\pi\in \Pi}n^{h-1}=\frac{\alpha^h}{n}B_{hN}
\end{eqnarray*}
where $B_{hN}$ is the Bell number, i.e. the number of partitions of the set $\{1,...,hN\}$. In particular, for $h\to \infty$ (see \cite{BeTa}) $ B_{hN}<\left(\frac{0.792hN}{\log(hN+1)}\right)^{hN}$, hence 
$$|R_2(n,h)|\leq \frac{|\alpha|^h}{n}\left(\frac{0.792hN}{\log(hN+1)}\right)^{hN}.$$
\end{proof}
\begin{remark}
These statistics are interesting when we consider how the processes are related to the 2-dimensional Wiener process, which has vanishing complex moments of all orders. The difference here is that the processes $W^n$ have unbounded 
variance as $n \to \infty$. If we rescale to $\bar W$ as in the introduction, then all moments vanish.

To give some intuition for why the $jN$ moments might be nonzero, consider the case $N = 4$. 
Roughly, the idea is that the process is more likely to be near one of the 4 rays in the directions $\{1, i, -1, -i\}$ 
than near the rays rotated by $\pi/4$. 
To see this imagine a time when ${\rm Re}\,W^n(t) \gg 1$. Then since ${\rm Im}\,W^n(t)$ is independent with mean 0, 
it is much more likely to be near 0 than it is to be near $\pm {\rm Re}\,W^n(t)$. 
Thus we see the underlying geometry of the processes is not statistically symmetric, as revealed in the moments; this is despite the fact that the paths converge to the fractal curves of the Wiener process, which {\em are} statistically symmetric.
\end{remark}

The next results show why the random walk $W_n$ can be regarded in a very weak sense as an $N-$stable process, in the sense of Theorem \ref{t1}.

\begin{lemma}\label{lemma-exp}
For any $\lambda,\alpha\in \bC$, $t\geq 0$ 
$$\bE[\exp(i \lambda W_n(t)]=\exp \left(\frac{i^N \alpha t}{N!} \lambda^N \right)+{\bf R}_n(\lambda),$$
where the remainder term ${\bf R}_n(\lambda)$ satisfies the following estimate
\begin{equation}\label{in-rem-exp}
|{\bf R}_n(\lambda)|\leq\left| \sum_{h=\lfloor nt\rfloor +1}^{+\infty}\frac{i ^{hN}\lambda^{hN} }{h!}\left(\frac{\alpha t}{N!}\right)^h \right|+ \frac{1}{n}\sum_{h=2}^\infty\frac{|\alpha|^h|\lambda|^{hN} }{(hN)!}\left(\frac{t^{h-1}(h^2+h)}{2}+\left(\frac{0.792hN}{\log(hN+1)}\right)^{hN}\right).
\end{equation}
\end{lemma}
\begin{proof}
\begin{eqnarray*}
& &\bE[\exp(i \lambda W_n(t)] =\lim_{m \to \infty}\sum_{k=0}^m\frac{1}{k!}i ^k\lambda^k \bE[(W_n(t))^k]\\
& &=\sum_{h=0}^{\lfloor nt\rfloor}\frac{i ^{hN}\lambda^{hN} }{(hN)!}\left(\frac{\alpha t}{N!}\right)^h\frac{(hN)!}{h!}+  \lim_{m \to \infty}\sum_{h=2}^m \frac{i ^{hN}\lambda^{hN} }{(hN)!}R(n,hN)\\
& &=\exp \left(\frac{i^N \alpha t}{N!} \lambda^N \right)- \sum_{h=\lfloor nt\rfloor +1}^{+\infty}\frac{i ^{hN}\lambda^{hN} }{h!}\left(\frac{\alpha t}{N!}\right)^h +\lim_{m \to \infty}\sum_{h=2}^m \frac{i ^{hN}\lambda^{hN} }{(hN)!}R(n,hN)
\end{eqnarray*}
where $R(n,hN)$ stands for  remainder term in Theorem \ref{t:unif}. The estimate \eqref{in-rem-exp} follows directly from the inequality \eqref{in-momenti}. 
\end{proof}
A direct consequence of Lemma \ref{lemma-exp} is the following result
\begin{theorem}
\label{t1}
The characteristic function $\psi_n(\lambda)$ of $W^n(t)$ satisfies
\begin{equation}
\lim_{n \to \infty} \bE[\exp(i \lambda W^n(t))] = \exp \left(i^N \frac{\lambda^N}{N!} \alpha t \right).
\end{equation}
\end{theorem}
\begin{proof}
The result, namely $\lim_{n \to \infty} {\bf R}_n(\lambda)=0$, follows by the convergence of the series 
$\sum \frac{i ^{hN}\lambda^{hN} }{h!}\left(\frac{\alpha t}{N!}\right)^h$ and the estimate
$\sum\limits_{h=2}^\infty\frac{|\alpha|^h|\lambda|^{hN} }{(hN)!}\left(\frac{t^{h-1}(h^2+h)}{2}+\left(\frac{0.792hN}{\log(hN+1)}\right)^{hN}\right)<\infty$. 
\end{proof}

\begin{theorem}\label{th-limit}
Let $f:\bC\to \bC$ be an entire analytic function with the power series expansion $f(z)=\sum\limits_{k=0}^\infty a_kz^k$, such that   
the coefficients $\{a_k\}_{k\in N\bN}$ satisfy the following assumption:
\begin{equation} \label{coefficient}
\sum_{h=0}^{\infty} |a_{hN}|\left(\frac{0.792hN}{\log(hN+1)}\right)^{hN}<\infty
\end{equation}
Then 
\begin{align*}
\lim_{n\to\infty}\bE[f(W_n(t))] = \sum_{h=0}^\infty a_{hN}\frac{(hN)!}{h!}\left(\frac{\alpha t}{N!}\right)^h
 = \sum_{h=0}^\infty \frac{f^{(hN)}(0)}{h!}\left(\frac{\alpha t}{N!}\right)^h.
\end{align*}
\end{theorem}

\begin{proof}
We directly compute
\begin{multline*}
\lim_{n \to \infty} \bE[f(W_n(t))]=\lim_{n \to \infty} \sum_{k=0}^\infty a_k\bE[(W_n(t))^k]
\\
=\lim_{n \to \infty}\sum_{h=0}^{\lfloor nt\rfloor}a_{hN}\left(\frac{\alpha t}{N!}\right)^h\frac{(hN)!}{h!}+\lim_{n \to \infty} \lim_{m \to \infty}\sum_{h=0}^m a_{hN}R(n,hN)=\sum_{h=0}^{\infty}a_{hN}\left(\frac{\alpha t}{N!}\right)^h\frac{(hN)!}{h!}
\end{multline*}
Indeed, by assumption \eqref{coefficient}, we have 
$$
  \left|\sum_{h=0}^m a_{hN}R(n,hN)\right|\leq \frac{C}{n}, \qquad C:=\sum_{h=0}^\infty a_{hN}|\alpha|^h\left(\frac{t^{h-1}(h^2+h)}{2}+\left(\frac{0.792hN}{\log(hN+1)}\right)^{hN}\right)<\infty.
$$
\end{proof}

\begin{lemma}\label{lemmacrescita}
If there exist $C_1,C_2\in\bR$ such that for all $k$ 
the coefficients $a_k$ satisfy the inequality $|a_k|\leq \frac{C_1\,C_2^k}{k!}$, then they satisfy assumption \eqref{coefficient}.
\end{lemma}

\begin{definition}\label{def-exp}
An analytic function $f$ is of {\bf exponential type} $c$ if
\begin{align*}
f(z) = \sum_{k=0}^\infty a_k z^k = \sum_{k=0}^\infty \frac{b_k}{k!} z^k
\end{align*}
where $|b_k|^{1/k} \to c$ as $k \to \infty$. Clearly if $c < \infty$ then $f$ is entire analytic. 
\end{definition}

Roughly, the idea is that a function of exponential type $c$ is asymptotically bounded by
$e^{c |z|}$. Alternatively, a necessary and sufficient condition is that $c = \limsup_{n \to \infty} |f^{(n)}(x)|$ for one (and hence all) $x \in \bC$.

%\begin{definition}
%An analytic function $f$ is said of {\em exponential type $c$} if 
%$f(z)=\sum_{k=0}^\infty a_k z^k$ and $(a_k k!)^{1/k}\to c$ as $k\to \infty$.
%\end{definition}

\begin{lemma}\label{lemma8}
If $f$ is of exponential type, then it satisfies assumption \eqref{coefficient}.
\end{lemma}

\begin{lemma}
If $f:\bC\to\bC$ is the Fourier transform of a complex bounded variation measure $\mu$ on $\bR$ with compact support, then $f$ satisfies the assumptions of Lemma \ref{lemmacrescita}.
\end{lemma}

\begin{corollary}
If $f:\bC\to\bC$ is the Fourier transform of a complex bounded variation measure $\mu$ on $\bR$ with compact support, then for all $t\in\bR, x\in \bR$
$$\lim_{n\to \infty }\bE[f(x+W_n(t))]=\int e^{iyx} e^{i^N\alpha t \frac{y^N}{N!}}d\mu(y)$$
\end{corollary}

\begin{remark}
By the Paley-Wiener Theorem \cite{Ru}, any function $f\in L^2(\bR)$ of exponential type is the Fourier transform of a function $\hat f\in L^2(\bR)$ with compact support.\\
More generally, any function $f:\bC\to\bC$  which is the Fourier transform of a complex bounded variation measure $\mu$ on $\bR$ with compact support is of exponential type, and furthermore it is bounded on the real line.
\end{remark}

\section{It\^o calculus}

We shall consider a sort of It\^o calculus for suitable regular functions of the random walk $W^n$ that mimics the development of classical stochastic differential equations.
We begin by considering integrals with respect to the processes $(W^n)^k$ for arbitrary $k \in \bN$.
Even if Theorem \ref{th-limit} and the estimate \eqref{in-momenti} would allow  the development of the theory for  a more general class of analytic functions $g$, as  stated in the introduction, we shall restrict to the case where $g$ is an  analytic function of exponential type. This is sufficient for our purposes and will simplify the notation. 
%\medskip

%Now we prove the main result in this section. Recall that Theorem \ref{t:unif} shows that power functions $g(z)$ applied to the random walk $W^n$ are well defined random variables $g(W^n(t))$. Next theorem extends this result to the class of analytic function of exponential type.

We state here, for later use, the following theorem, which follows directly from Theorem \ref{th-limit} and Lemma \ref{lemma8}. 
\begin{theorem}\label{t:163}
Let $g(z)  = \sum_{k=0}^\infty a_k z^k$ be an analytic function of exponential type $c$.
Then
\begin{equation}
\label{e:163-0}
\lim_{n \to \infty} \bE[g(W^n(t))] = \sum_{h=0}^\infty a_{hN}\frac{(hN)!}{h!}\left(\frac{\alpha t}{N!}\right)^h
 =\sum_{h=0}^\infty \frac{g^{(hN)}(0)}{h!}\left(\frac{\alpha t}{N!}\right)^h.
\end{equation}
\end{theorem}

We define an {\bf It\^o integral} for the process $g(W^n(t))$ as
\begin{align*}
\int_0^t g(W^n(s)) \, {\rm d}W^n(s) = \sum_{\tau = 0}^{\lfloor n t \rfloor - 1} g(W^n(\tfrac{\tau}{n})) (W^n(\tfrac{\tau+1}{n}) - W^n(\tfrac{\tau}{n})) = \frac{1}{n^{1/N}} \, \sum_{\tau = 0}^{\lfloor n t \rfloor - 1} g(W^n(\tfrac{\tau}{n})) (\xi_{\tau+1})
\end{align*}
and, for any $k \in \bN$,
\begin{align}
\label{e:ito-k}
\int_0^t g(W^n(s)) \, {\rm d}(W^n(s))^k = \frac{1}{n^{k/N}} \, \sum_{\tau = 0}^{\lfloor n t \rfloor - 1} g(W^n(\tfrac{\tau}{n})) (\xi_{\tau+1})^k.
\end{align}

Our next step is the analysis of expectations. 
Taking the mean in both sides of \eqref{e:ito-k} and recalling that $W^n(\tfrac{\tau}{n})$ is independent from $\xi_{\tau+1}$, we get
\begin{align*}
\bE \left[\int_0^t g(W^n(s)) \, {\rm d}(W^n(s))^k\right] = \frac{1}{n^{k/N}} \, \sum_{\tau = 0}^{\lfloor n t \rfloor - 1} \bE[g(W^n(\tfrac{\tau}{n}))] \, \bE[ (\xi_{\tau+1})^k]
\end{align*}
and recalling \eqref{e3} we get 
\begin{align}
\label{e.zero-mean}
\bE \left[\int_0^t g(W^n(s)) \, {\rm d}(W^n(s))^k\right] = 0 \qquad \text{for all $k \not= m N$, $m \in \bN$;}
\end{align}
for $k=N$ we have
\begin{align*}
\bE \left[\int_0^t g(W^n(s)) \, {\rm d}(W^n(s))^N\right] = \alpha \int_0^t \bE[g(z + W^n(s))] \, {\rm d}s
\end{align*}
and finally for $k = m N$, $m \in \bN$, $m > 1$:
\begin{align}\label{eq:unify-1}
\bE \left[\int_0^t g(W^n(s)) \, {\rm d}(W^n(s))^{m N}\right] = \left(\frac{\alpha}{n}\right)^m \sum_{\tau=0}^{\lfloor n t\rfloor-1} \bE[g(W^n(\tfrac{\tau}{n}))] = \alpha \, \left(\frac{\alpha}{n}\right)^{m-1} \int_0^t \bE[g(W^n(s))] \, {\rm d}s
\end{align}
so we expect this to vanish to zero as $n \to \infty$ when $m > 1$.

\begin{lemma}
\label{le.1}
Assume that $g$ is an analytic function of exponential type $c$. Then
\begin{align*}
\lim_{n \to \infty} \int_0^t \bE[g(W^n(s))] \, {\rm d}s =  \int_0^t \lim_{n \to \infty} \bE[g(W^n(s))] \, {\rm d}s.
\end{align*}
\end{lemma}

\begin{proof}%[Proof of Lemma \ref{le.1}]
From Theorem \ref{t:163} we have
\begin{align*}
\bE[g(W_n(t))]=
g_1(t) + R(n,t)
%\left(\frac{\alpha t}{N!} \right)^m \frac{(mN)!}{m!} + R_n(t)
\end{align*}
where 
\begin{align*}
g_1(t) = \sum_{h=0}^\infty \frac{b_{hN}}{h!} \left(\frac{\alpha t}{N!} \right)^h
\end{align*}
and
$|R_n(t)| \le \frac{1}{n} C(T,\alpha)$.
The claim now follows from an application of Dominated Convergence Theorem.
\end{proof}

\medskip

In particular, we may record the following identity, concerning the limit behavior for polynomials. It
follows from Lemma \ref{le.1} and a direct application of Theorem \ref{t:unif}.

\begin{corollary}\label{c:unif}
Let $g(x) = x^{mN}$. Then
\begin{align*}
\lim_{n \to \infty} \int_0^t \bE[g(W^n(s))] \, {\rm d}s =  \frac{(mN)!}{(m+1)!} \left(\frac{\alpha}{N!}\right)^m t^{m+1}.
\end{align*}
\end{corollary}

\begin{proof}
From Theorem \ref{t:unif} we have
\begin{align*}
\bE[(W_n(t))^{m N}]=\left(\frac{\alpha t}{N!} \right)^m \frac{(mN)!}{m!} + R_n(t)
\end{align*}
where $|R_n(t)| \le \frac{1}{n} C(T,m,\alpha)$.
%The claim now follows from an application of Uniform Convergence Theorem.
%
By Lemma \ref{le.1} we get
\begin{align*}
\lim_{n \to \infty} \int_0^t \bE[g(W^n(s))] \, {\rm d}s =  \int_0^t \lim_{n \to \infty} \bE[g(W^n(s))] \, {\rm d}s
\end{align*}
and the right hand side is equal to
\begin{align*}
\frac{(mN)!}{m!} \int_0^t \left(\frac{\alpha s}{N!}\right)^m \, {\rm d}s = \frac{(mN)!}{(m+1)!} \left(\frac{\alpha}{N!}\right)^m t^{m+1}
\end{align*}
as required.
\end{proof}

\begin{corollary}\label{c:213}
Assume that $g$ is an analytic function of exponential type $c$. Then
\begin{align*}
\lim_{n \to \infty} \int_0^t \bE[g(W^n(s))] \, {\rm d}s =  \sum_{h=0}^\infty \frac{b_{hN}}{(h+1)!} \left(\frac{\alpha}{N!} \right)^h t^{h+1}.
\end{align*}
\end{corollary}

\begin{corollary}\label{c:163}
For any $k \not= N$ it follows that
\begin{align*}
\lim_{n \to \infty} \bE \left[\int_0^t g(W^n(s)) \, {\rm d}(W^n(s))^{k}\right] = 0.
\end{align*}
\end{corollary}

\section{Diffusions}

\begin{theorem}\label{t:9}
Let $g: \bC \to \bC$ be of exponential type $c < \infty$.
Then
the following {\bf It\^o formula} holds
\begin{align}\label{e:gif}
g(z + W^n(t)) - g(z) = \sum_{k=1}^\infty \frac{1}{k!} \int_0^t \partial^k g(z + W^n(s)) \, {\rm d}(W^n(s))^k.
\end{align}
\end{theorem}

\begin{proof}
For simplicity, we let $t = \theta/n$, $\theta \in \bN$.
By exploiting a telescopic sum and Taylor's expansion of $g$ we get
\begin{align*}
g(z + W^n(t)) - g(z) = \sum_{\tau=0}^{\theta-1} g(z+W^n(\tfrac{\tau+1}{n})) - g(z+W^n(\tfrac{\tau}{n}))
= \sum_{\tau=0}^{\theta-1} \sum_{k=1}^\infty \frac{1}{k!} \partial^{k}g(z + W^n(\tfrac{\tau}{n})) \, \frac{(\xi_{\tau+1})^k}{n^{k/N}}
\end{align*}
and the representation \eqref{e:gif} follows by the interchange of sums which is allowed due to the absolute convergence of the Taylor series.
\end{proof}

\begin{remark}
The It\^o's formula for Brownian motion satisfies 
\begin{align*}
\bE[g(z + B(t))] - g(z) = \frac12 \int_0^t \bE[\partial^2 g(z + B(s))] \, {\rm d}s
\end{align*}
which we can restate in terms of the random walk, as
\begin{align}
\label{e:sb1}
\lim_{n \to \infty} \bE[g(z + W^n(t))] - g(z) = \frac12 \int_0^t \lim_{n \to \infty} \bE[ \partial^2 g(z + W^n(s))] \, {\rm d}s.
\end{align}
We aim to prove that (a suitable extension of) \eqref{e:sb1} holds for the $N$-th order differential operator with respect to the
random walk on the complex plane defined on the lattice generated by $R(N)$.
\end{remark}

Let us state  the aim of our construction.

\begin{theorem}
\label{t:main-claim}
Assume that $g$ is an analytic function of exponential type $c$, i.e., $|g^{(k)}(z)|^{1/k} \to c$ as $k \to \infty$ for every $z \in \bC$.
Then we have
\begin{align}\label{eq:unify1}
\lim_{n \to \infty} \bE[g(z + W^n(t))] - g(z) = \frac{\alpha}{N!}  \int_0^t \lim_{n \to \infty} \bE[\partial^{N} g(z + W^n(s))] \, {\rm d}s.
\end{align}
Setting
\begin{align*}
u(t,z) = \lim_{n \to \infty} \bE[g(z + W^n(t))],
\end{align*}
this is equivalent to say that $u(t,z)$ is a classical solution of the $N$-th order Cauchy problem

\begin{equation}
\begin{aligned}
&\partial_t u(t,z) = \frac{\alpha}{N!} \partial^N u(t,z), %\qquad &t \in [0,\infty),
\\
&u(0,z) = g(z),% \qquad \phantom{\frac{\partial^N}{\partial x^N}}&z \in \bC.
\end{aligned}
\label{heat-N}
\end{equation}
\end{theorem}

\begin{proof}
By assumption, we can start from the It\^o's formula \eqref{e:gif}; by taking the expectation in both sides we get
\begin{align}\label{eq:unify0}
\bE[g(z + W^n(t)) - g(z)] = \sum_{m=1}^\infty \frac{\alpha}{(mN)!} \left(\frac{\alpha}{n}\right)^{m-1} \int_0^t \bE[\partial^{mN} g(z + W^n(s))] \, {\rm d}s
\end{align}
then we take the limit as $n \to \infty$ and apply Corollary \ref{c:163} to obtain \eqref{eq:unify1}.
\\
The second claim of the theorem follows once we prove that
\begin{align*}
\lim_{n \to \infty} \bE[\partial^{N} g(z + W^n(s))] = \partial^{N} \lim_{n \to \infty} \bE[g(z + W^n(s))]
\end{align*}
Notice that, by Theorem \ref{t:163}, we can write 
\begin{align*}
\lim_{n \to \infty} \bE[\partial^{N} g(z + W^n(s))] = \sum_{h=0}^\infty \frac{\partial^{hN}g^{(N)}(z)}{h!} \left(\frac{\alpha t}{N!} \right)^h
= \partial^N \sum_{h=0}^\infty \frac{\partial^{hN}g(z)}{h!} \left(\frac{\alpha t}{N!} \right)^h
= \partial^{N} \lim_{n \to \infty} \bE[g(z + W^n(s))]
\end{align*}
so \eqref{heat-N} follows.
\end{proof}
For a generalization of these results to the case of the boundary value problem associated with \eqref{heat-N} on a bounded interval with Dirichelet, resp. Neumann, resp. periodic boundary conditions see also  \cite{Bonaccorsi2015}.

\subsection{It\^o formula for polynomials}

Assume for this section that $g(z) = c_\beta \, z^\beta$ is a polynomial of order $\beta = b N$.
Then we can calculate some examples using the extended It\^o's formula \eqref{eq:unify1}.
%\begin{align}\label{eq:unify}
%\lim_{n \to \infty} \bE[g(z + W^n(t))] - g(z) = \frac{\alpha}{N!} \int_0^t \lim_{n \to \infty} \bE[\partial^{N} g(z + W^n(s))] \, {\rm d}s.
%\end{align}
%Therefore
Actually, we have %\footnote{N.d.S.: details are hidden in the source file}
\begin{align*}
\lim_{n \to \infty} \bE[g(z + W^n(t))] - g(z) &= c_{\beta} \, \frac{(bN)!}{((b-1)N)!} \frac{\alpha}{N!}  \int_0^t \lim_{n \to \infty} 
\sum_{j=0}^{b-1} \binom{(b-1)N}{jN} z^{(b-1-j)N} \, \bE[W^n(s)^{jN}]
\, {\rm d}s
%\\
%&= c_{\beta} \, \frac{(bN)!}{((b-1)N)!} \frac{\alpha}{N!} \sum_{j=0}^{b-1} \frac{((b-1)N)!}{(jN)! ((b-1-j)N)!}  \int_0^t \lim_{n \to \infty}  
%z^{(b-1-j)N} \, \bE[W^n(s)^{jN}]
%\, {\rm d}s
%\\
%&= c_{\beta} \, \frac{\alpha}{N!} \sum_{j=0}^{b-1} \frac{(bN)!}{(jN)! ((b-1-j)N)!}  \int_0^t 
%z^{(b-1-j)N} \, \left(\frac{\alpha s}{N!} \right)^j \frac{(jN)!}{j!}
%\, {\rm d}s
%\\
%&= c_{\beta} \,  \frac{\alpha}{N!} \sum_{j=0}^{b-1} \frac{(bN)!}{((b-1-j)N)!}  
%z^{(b-1-j)N} \, \left(\frac{\alpha}{N!} \right)^j \frac{1}{(j+1)!}
%t^{j+1}
%\\
%&=c_{\beta} \,  \sum_{j=0}^{b-1} \frac{(bN)!}{[(b-1-j)N]!}  
%z^{(b-1-j)N} \, \left(\frac{\alpha}{N!} \right)^{j+1} \frac{t^{j+1}}{(j+1)!}
\\
&=c_{\beta} \,  \sum_{j=1}^{b} \frac{(bN)!}{[(b-j)N]!}  
z^{(b-j)N} \, \left(\frac{\alpha}{N!} \right)^{j} \frac{t^{j}}{j!}.
\end{align*}
In the special case $\beta = N$ the It\^o's formula \eqref{e:gif} gives
\begin{align*}
(W^n(t))^N %= \sum_{k=1}^N \frac{1}{k!} \int_0^t  \frac{N!}{(N-k)!} (W^n(s))^{N-k} \, {\rm d}(W^n(s))^k.
 = \sum_{k=1}^N \binom{N}{k} \int_0^t   (W^n(s))^{N-k} \, {\rm d}(W^n(s))^k.
\end{align*}

\subsection{It\^o formula for Wiener-type integrals}

Let $\varphi: [0,T] \to \bC$ be a continuous and bounded function.
Then we can define the Wiener integral
\begin{equation}\label{Wiener1}
\int _0^t \varphi(s) \, {\rm d}W_n(s) =
\sum_{\tau=0}^{\lfloor nt \rfloor - 1} \varphi(\tfrac{\tau}{n}) (W_n(\tfrac{\tau+1}{n}) - W_n(\tfrac{\tau}{n})) =
\frac{1}{n^{1/N}} \sum_{\tau=0}^{\lfloor nt \rfloor - 1} \varphi(\tfrac{\tau}{n}) \xi_{\tau+1}
\end{equation}
and, for any $k \ge 1$,
\begin{equation}\label{Wiener1k}
\int _0^t \varphi(s) \, {\rm d}(W_n(s))^k =
\frac{1}{n^{k/N}} \sum_{\tau=0}^{\lfloor nt \rfloor - 1} \varphi(\tfrac{\tau}{n}) (\xi_{\tau+1})^k.
\end{equation}
These formulas define, for ``sufficiently regular" $\varphi: \bR \to \bC$, a family of stochastic processes $X_{k,n}(t)=\int _0^t \varphi(s) \, {\rm d}(W_n(s))^k$
which extend the construction of the previous section.

%Analogously, 
%for $U:\C\to\C$, the Ito integral $\int _0^t U(W_n(s))dW_n(s)$ is defined as:
%\begin{equation}\label{Ito1}
%\int _0^t U(W_n(s)) dW_n(s):=\lim_{|P|\to 0}  \sum_{i=1}^mU(W(\tau_{i-1}))(W^n_{\tau_i}-W^n_{\tau_{i-1}}).\end{equation}
%
%As for fixed $n$ the process $W^n$ is a jump process, with jumps occuring at fixed time intervals $t_i=i/n$, with $i\in \N$, then the limits in \eqref{Wiener1} and \eqref{Ito1} exist and are equal to 
% \begin{equation}\label{Wiener2}
%\int _0^tg(s)dW_n(s):=  \sum_{i=1}^{\lfloor n t \rfloor+1}g(t_{i-1})(W^n_{t_i}-W^n_{t_{i-1}}),
%\end{equation}
%and to 
%\begin{equation}\label{Ito2}
%\int _0^t U(W_n(s)) dW_n(s):= \sum_{i=1}^{\lfloor n t \rfloor +1}U(W^n(t_{i-1}))(W^n_{t_i}-W^n_{t_{i-1}}),
%\end{equation}
%where $t_i=i/n$, $i\in\N$.

%
%\begin{remark}
%Here we have to state precise conditions on the functions $\varphi$.
%\end{remark}

We have the following application, which generalizes Theorem 11 in \cite{Bonaccorsi2015}.

\begin{theorem}\label{th-18}
Let us consider the initial value problem
\begin{equation}
\begin{aligned}
&\partial_t u(t,x) = \frac{\alpha}{N!} (\varphi(t))^N \partial^N_x u(t,x), \qquad &t \in [0,\infty),
\\
&u(0,x) = f(x), \qquad \phantom{\frac{\partial^N}{\partial x^N}}&x \in \bR.
\end{aligned}
\label{e1}
\end{equation}
where $\varphi\in C_b([0,T];\bC)$ and $f:\bR\to\bC$ is an analytic function of exponential type.
%\begin{align*}
%f(x) = \int_{\R} e^{i x y} \, {\rm d}\mu(y),
%\end{align*}
%where $\mu$ is a measure of bounded variation on $\R$ satisfying the following assumptions:
%\begin{enumerate}
%\item[1.] $\displaystyle \int_{\R} |e^{i x z}| \, {\rm d}|\mu|(x) < \infty$ for all $z \in \C$,
%\item[2.] there exists a time interval $(T_1,T_2)$, with $T_1 < t_0 < T_2 \in \R$, such that
%\begin{align*}
% \int_{\R} \left|\exp\left(i^N \alpha  \frac{x^N}{N!}\int_{t_0}^tg(s)^Nds)\right)\right| \, {\rm d}|\mu|(x) < \infty
% \end{align*}
%  for all $t \in (T_1,T_2)$.
%\end{enumerate}

Then the function $u(t,x)$ defined by
\begin{equation}
u(t,x)=\lim_{n\to\infty} \bE\left[ f \left(x+\int_{0}^t \varphi(s) \, {\rm d}W^N_n(s) \right)\right] 
\end{equation}
for $t \in [0,T]$ is a classical solution of the parabolic problem \eqref{e1}.
% for any time $t \in (T_1,T_2)$ in the sense that 
%\begin{equation}\label{solHoc}
%u(t,x)= \int_{\bR} e^{i\, x\, y} \exp\left(\frac{i^N \alpha}{N!} y^N\int_{t_o}^tg(s)ds \right) \, {\rm d}\mu(y)
%\end{equation}
\end{theorem}

\begin{proof}
We may proceed as in the proof of Theorem \ref{t:main-claim}.
Let us consider the process $X_n(t) = \int_0^t \varphi(s) \, {\rm d}W_n(s)$; we search for an It\^o formula for the process $f(z+X_n(t))$.
Notice that
\begin{align*}
f(z + X_n(t)) - f(z) = \sum_{\tau=0}^{\theta-1} f(z+X_n(\tfrac{\tau+1}{n})) - f(z+X_n(\tfrac{\tau}{n}))
= \sum_{\tau=0}^{\lfloor nt \rfloor -1} \sum_{k=1}^\infty \frac{1}{k!} \partial^{k}f(z + X_n(\tfrac{\tau}{n})) \, \frac{(\varphi(\tfrac{\tau}{n}) \, \xi_{\tau+1})^k}{n^{k/N}}
\end{align*}
and exchanging the order of the sums, we get
\begin{align}
f(z + X_n(t)) - f(z) = \sum_{k=1}^\infty \frac{1}{k!} \int_0^t \left(\varphi(s)\right)^k \, \partial^k f(z + X_n(s)) \, {\rm d}(W^n(s))^k.
\end{align}
Taking the mean in both sides of previous formula we get
\begin{align}\label{e:0402-1}
\bE[f(z + X_n(t)) - f(z)] = \sum_{k=1}^\infty \frac{1}{k!} \bE\left[ \int_0^t \left(\varphi(s)\right)^k \, \partial^k f(z + X_n(s)) \, {\rm d}(W^n(s))^k \right]
\end{align}
and proceeding as in the analysis of \eqref{e:ito-k}, and recalling \eqref{e3}, we get 
\begin{align*}
\bE \left[\int_0^t (\varphi(s))^k \partial^k f(z + X_n(s)) \, {\rm d}(W_n(s))^k\right] = \frac{1}{n^{k/N}} \, \sum_{\tau = 0}^{\lfloor n t \rfloor - 1} \varphi(s) \, \bE[\partial^k f(z + X_n(s))] \,  \bE[ (\xi_{\tau+1})^k]
\end{align*}
hence we get
\begin{align}
%\label{e.zero-mean}
\bE \left[\int_0^t (\varphi(s))^N \, \partial^k f(z + X_n(s)) \,  {\rm d}(W_n(s))^k\right] = 0 \qquad \text{for all $k \not= m N$, $m \in \bN$;}
\end{align}
for $k=N$ we have
\begin{align*}
\bE \left[\int_0^t \varphi(s) \, \partial^N f(z + X_n(s)) \,  {\rm d}(W_n(s))^N\right] = \alpha \int_0^t (\varphi(s))^N \, \bE[\partial^N f(z + X_n(s))] \, {\rm d}s
\end{align*}
and finally for $k = m N$, $m \in \bN$, $m > 1$:
\begin{align}
%\label{eq:unify-1}
\bE \left[\int_0^t (\varphi(s))^{mN} \, \partial^k f(z + X_n(s)) \,  {\rm d}(W_n(s))^{m N}\right] = \alpha \, \left(\frac{\alpha}{n}\right)^{m-1} \int_0^t (\varphi(s))^{mN} \,  \bE[\partial^{mN} f(z + X_n(s))] \, {\rm d}s.
\end{align}
We notice that this term is asymptotically dominated in $m$ by $\left(\frac{\alpha}{n}\right)^{m-1} t \|\varphi\|_\infty^{mN} c^{mN}$, $c$ being the exponential type of $f$.
\\
We then see that the series on the right-hand side of \eqref{e:0402-1} is dominated by a convergent series
\begin{align*}
\sum_{m=1}^\infty \frac{1}{(mN)!} \left(\frac{\alpha}{n}\right)^{m-1} t \|\varphi\|_\infty^{mN} (c+\varepsilon)^{mN}
\end{align*}
where $c$ is the exponential type of $f$ and $\varepsilon > 0$;
hence it is possible to pass to the limit as $n \to \infty$ inside the sum, to get
\begin{align*}
\lim_{n \to \infty} \bE[f(z + X_n(t)) - f(z)] =  \alpha \int_0^t (\varphi(s))^N \, \lim_{n \to \infty} \bE[\partial^N f(z + X_n(s))] \, {\rm d}s.
\end{align*}

\end{proof}

%\begin{remark} 
%In order to generalize completely theorem 11 in \cite{Bonaccorsi2015} we have to state precise assumptions on $g$ and $f$. Moreover we also should say when \eqref{solHoc} is a classical solution of the parabolic problem \eqref{e1}.
%\end{remark}
%

%\vfill\break

\section{The Feynman-Kac formula}\label{sez-FK}
The  results of the previous  section, in particular Theorem \ref{t:main-claim} and Theorem \ref{th-18} allow the proof of a probabilistic representation for the solution of heat-type equation of order $N>2$.
In this section we generalize the second statement of Theorem \ref{t:main-claim}  to the case where a time dependent potential  is added, 
 proving a Feynman-Kac type formula for the perturbed problem 
\begin{equation}
\begin{aligned}
&\frac{\partial}{\partial t}u(t,x) = \frac{\alpha}{N!} \frac{\partial^N}{\partial x^N}u(t,x)+V(t,x)u(t,x), \qquad &t \in [0,\infty),
\\
&u(0,x) = f(x), \qquad \phantom{\frac{\partial^N}{\partial x^N}}&x \in \bR.
\end{aligned}
\label{heat-N-V}
\end{equation}
i.e. a probabilistic representation of the form 
\begin{equation}\label{Fey-Kac}
u(t,x)=\lim_{n\to \infty}\bE\left[f(x+W_n(t))e^{\int_0^t V(t-s, x+W_n(s))ds}\right]
\end{equation}
Since the sequence of processes $\{W_n(t)\}$ does not properly converge, the existence of the limit on the right hand side of Eq. \eqref{Fey-Kac} is not assured, even for smooth and bounded potential function $V$. 
Indeed the result of section 2 allow to prove the "asymptotic integrability" of cylinder functions of the form $W_n\mapsto  f(W_n(t))$, i.e. the existence of the limit $\lim_{n\to \infty }\bE[f(W_n(t))]$, for $f$ entire analytic satisfying the assumptions of Theorem \ref{th-limit}.  
The construction of formula \eqref{Fey-Kac} requires in fact the proof of the integrability of more general (non-cylinder) functions of the form $W_n\mapsto  \exp\left(    \int_0^t f(s,W_n(s)) \, {\rm d}s\right)$. 

\begin{theorem}\label{theor-cyl}
Let $a:\in L^1([0,t])$. Then 
    $$\lim_{n\to \infty}\bE\left[e^{\int_0^t a(s)W_n(s)ds}\right]=\exp\left(\frac{\alpha}{N!}\int_0^t(\int_s^t a(u) du)^Nds\right)$$
\end{theorem}
\begin{proof}
Set $t_j:=\frac{j}{n}$, for $j=0,..., \lfloor nt\rfloor$
\begin{eqnarray*}
\bE\left[e^{\int_0^t a(s)W_n(s)ds}\right]&=&  \bE\left[e^{\sum_{j=0}^{\lfloor nt\rfloor-1} W_n(t_j)\int_{t_j}^{t_{j+1}}a(s)ds}\right]\\
&=& \bE\left[e^{\sum_{j=1}^{\lfloor nt\rfloor}\alpha^{1/N}n^{-1/N}\xi_j\int_{t_{j-1}}^{t}a(s)ds}\right]=\prod_{j=1}^{\lfloor nt\rfloor}\bE\left[   \exp\left(\alpha^{1/N}n^{-1/N}\xi_j\int_{t_{j-1}}^{t}a(s)ds)\right)\right]\\
&=&\prod_{j=1}^{\lfloor nt\rfloor}\left( e^{\frac{\alpha}{N!n}(\int_{t_{j-1}}^{t}a(s)ds)^N}+R(n,j))\right)\\
&=&e^{\sum_{j=1}^{\lfloor nt\rfloor}\frac{\alpha}{N!n}(\int_{t_{j-1}}^{t}a(s)ds)^N}+{\mathcal R}_n 
\end{eqnarray*}
where, by the estimate in Lemma \ref{stime-chi} and the boundedness of the function $s\mapsto \int_s^t a(\tau )d\tau$
\footnote{Indeed $ |\int_s^t a(\tau )d\tau|\leq \|a\|_{L^1([0,t])}$}, we have that there exists a constant $C\in \bR$ such that for all $j=1,..., \lfloor nt\rfloor$ the inequality 
 $|R(n,j)|\leq C/n^2$ holds.
Hence, consequently, 
$$|{\mathcal R}_n|\leq \left(1+\frac{C'}{n^2}\right)^n-1, $$
where $C'=Ce^{\frac{|\alpha|}{N!n}\|a\|_{L^1([0,t])}^N}$
By taking the limit for $n\to \infty $ we get the thesis.
\end{proof}

Let us consider now the initial value problem \eqref{heat-N-V} with a time dependent potential which is linear in the space variable, i.e. $V:[0,t]\times \bR\to \bR$ is of the form
$V(\tau,x)=A(\tau)x$, where $A:[0,+\infty)\to \bR$ is a continuous function. In this case the results of Theorem \ref{theor-cyl} allow us to prove the Feynman-Kac formula \eqref{Fey-Kac}.
\begin{theorem}
Let  $A:[0,+\infty)\to \bR$ be a continuous function and $f:\bR\to \bC$ the Fourier transform of a complex Borel measure on $\bR$  with compact support. Then the classical solution of the initial value problem 
\begin{equation}
\begin{aligned}
&\frac{\partial}{\partial t}u(t,x) = \frac{\alpha}{N!} \frac{\partial^N}{\partial x^N}u(t,x)+A(t)xu(t,x), \qquad &t \in [0,\infty),
\\
&u(0,x) = f(x), \qquad \phantom{\frac{\partial^N}{\partial x^N}}&x \in \bR.
\end{aligned}
\label{heat-N-V-lin}
\end{equation}
is given by
$$u(t,x)=\lim_{n\to\infty}\bE\left[f(x+W_n(t))e^{\int_0^tA(t-s)(W_n(s)+x)ds}   \right].$$
\end{theorem}
\begin{proof}
By considering the function $a:[0,t]\to \bR$ defined as $a(s):=A(t-s)$, $s\in [0,t]$ and by representing the function $f:\bR\to \bC$ in the form $f(x)=\int_\bR e^{iyx}d\mu_f(y)$, with $\mu $ Borel measure on $\bR$ with compact support, we obtain
\begin{eqnarray*}
& &\lim_{n\to\infty}\bE\left[f(x+W_n(t))e^{\int_0^tA(t-s)(W_n(s)+x)ds} \right]\\
&=&\lim_{n\to\infty}\bE\left[\int_\bR e^{iyx+iyW_n(t)}e^{\int_0^tA(t-s)W_n(s)ds}e^{x\int_0^tA(t-s)ds}d\mu (y) \right]\\
&=&e^{x\int_0^tA(t-s)ds}\lim_{n\to\infty}\int_\bR e^{iyx}\bE\left[ e^{iyW_n(t)}e^{\int_0^tA(t-s)W_n(s)ds}\right]d\mu (y) \\
\end{eqnarray*}
By Theorem \ref{theor-cyl} and  dominated convergence, the latter line converges to
$$e^{x\int_0^tA(t-s)ds}\int_\bR e^{iyx}e^{\frac{\alpha}{N!}\int_0^t (iy+\int_s^t A(t-u)du)^Nds}d\mu (y),$$
which is, as one can directly verify, the classical solution of the Cauchy problem \eqref{heat-N-V-lin}.
\end{proof}

\section{Stopping times}

Now replace $t$ with a stopping time $\tau_n$ for the stochastic process $W^n$ in the above formulas, with $\tau_n$ finite a.s.\ Clearly
\begin{align*}
H^{n,k}_t := \int_0^t g(W^n(s)) \, {\rm d}(W^n(s))^k = \sum_{j=0}^{\lfloor n t \rfloor - 1} g(W^n(\tfrac{j}{n})) \left(\frac{\xi_{j+1}}{n^{1/N}}\right)^k
\end{align*}
is a martingale for $k \not= mN$. To see this calculate $\bE(H^{n,k}_t \mid \cF_s) = H^{n,k}_s$ using the fact that
$H^{n,k}_t = (H^{n,k}_t - H^{n,k}_s) + H^{n,k}_s$ and
$\bE[H^{n,k}_t - H^{n,k}_s] = 0$
due to \eqref{e.zero-mean}.
\\
Thus the stopped process $H^{n,k}_{\tau_n \wedge t}$ is a martingale, and by the Optional Stopping Theorem (also called Doob's Optional Sampling Theorem) $\bE(H^{n,k}_{\tau_n \wedge t}) = \bE(H^{n,k}_0)$. Thus
\begin{align}\label{e:183-1}
\bE \int_0^{\tau_n} g(z + W^n(s)) \, {\rm d}(W^n(s))^k = 0, \qquad \text{for $k \not= mN$.}
\end{align}

According to It\^o's formula \eqref{e:gif} we have
\begin{align*}
g(z + W^n(\tau_n)) - g(z) = \sum_{k=1}^\infty \frac{1}{k!} \int_0^{\tau_n} \partial^k g(z + W^n(s)) \, {\rm d}(W^n(s))^k.
\end{align*}
Assume that $\tau_n$ is bounded $\bP$-a.s. (for instance, $\tau_n \le T$).
Taking the expectation of both sides of the previous identity 
we get
\begin{align}
\bE[g(z + W^n(\tau_n))] - g(z) =& \bE \sum_{k=1}^\infty \frac{1}{k!} \int_0^{\tau_n} \partial^k g(z + W^n(s)) \, {\rm d}(W^n(s))^k
\nonumber %\\
\intertext{and since $\bE$ is a finite sum}
=& \sum_{k=1}^\infty \frac{1}{k!}  \bE \int_0^{\tau_n} \partial^k g(z + W^n(s)) \, {\rm d}(W^n(s))^k %(\footnotemark)
\nonumber
\intertext{which thanks to \eqref{e:183-1} is}
=& \sum_{k=1}^\infty \frac{1}{(kN)!}  \bE \int_0^{\tau_n} \partial^{kN} g(z + W^n(s)) \, {\rm d}(W^n(s))^{kN} %(\footnotemark)
\nonumber\\
=& \sum_{k=1}^\infty \frac{\alpha^k}{(kN)!} \frac{1}{n^{k-1}} \bE \int_0^{\tau_n} \partial^{kN} g(z + W^n(s)) \, {\rm d}s
\nonumber
\end{align}
%
%
%=& \sum_{k=1}^\infty \frac{1}{k!}  \int_0^T \bE \int_0^{t} \partial^k g(z + W^n(s)) \, {\rm d}(W^n(s))^k \, \bP(\tau_n \in {\rm d}t) 
%\nonumber\\
%=& \lim_{n \to \infty} \sum_{h=1}^\infty \frac{1}{(hN)!}  \int_0^T \int_0^{t} \frac{1}{n^{h-1}} \bE [\partial^{hN} g(z + W^n(s))] \, {\rm d}s \, \bP(\tau_n \in {\rm d}t)
%\nonumber\\
%=& \lim_{n \to \infty} \sum_{h=1}^\infty \frac{1}{(hN)!}  \bE \int_0^{\tau_n} \frac{1}{n^{h-1}}  [\partial^{hN} g(z + W^n(s))] \, {\rm d}s 
%\nonumber\\
%=&  \frac{\alpha}{N!}  \lim_{n \to \infty} \bE \int_0^{\tau_n}   [\partial^{N} g(z + W^n(s))] \, {\rm d}s (\footnotemark).
%\label{e:183-2}
%\end{align}
%\footnotetext{because $\bE$ is a finite sum}
%\footnotetext{thanks to \eqref{e:183-1}}
%\footnotetext{by using the boundedness of $\tau$ and Corollary \ref{c:163}, just like in the proof of Theorem \ref{t:main-claim}} 
and passing to the limit as $n \to \infty$ we get
\begin{align}
\lim_{n \to \infty} \bE[g(z + W^n(\tau_n))] - g(z) =& \lim_{n \to \infty} \sum_{h=1}^\infty \frac{\alpha^h}{(hN)!} \frac{1}{n^{h-1}} \bE \int_0^{\tau_n} \partial^{hN} g(z + W^n(s)) \, {\rm d}s 
\nonumber\\
=&  \frac{\alpha}{N!}  \lim_{n \to \infty} \bE \int_0^{\tau_n}   [\partial^{N} g(z + W^n(s))] \, {\rm d}s (\footnotemark).
\label{e:183-2}
\end{align}
\footnotetext{by using the boundedness of $\tau$ and Corollary \ref{c:163}, just like in the proof of Theorem \ref{t:main-claim}}

\begin{remark}
Formula \eqref{e:183-2} is a generalization of the fundamental theorem of calculus. For $N = 1$ and $\alpha = 1$ the $W^n$ are deterministic 
and converge to the line $W_t^\infty = t$. Then \eqref{e:183-2} reads
\begin{align*}
g(z+t) - g(z) = \int_0^t g'(z+s) \, {\rm d}s
\end{align*}
for $z+t$ being the first hitting point to the boundary along the line to the right of $z$.

When $N = 2$ the process $W^n$ converges to the 1-dimensional Wiener process $W$ and $\tau_n$ converges to the 
Wiener process hitting time $\tau$ and we have
\begin{align*}
\bE [g(z + W_\tau)] - g(z) = \frac12 \bE \int_0^\tau g''(z + W_s) \, {\rm d}s
\end{align*}
so the average value of the integral of $g''$ from $z$ to the boundary is the difference of the second antidervative averaged on the boundary and at $z$. 
Notice $W_\tau$ has only 2 values, the hitting points to the right and left of $z$, since $W$ is the 1- dimensional Wiener process. 
E.g., in the particular, symmetric case of $U$ being the disc centered at $z$ of radius $r$, we get
\begin{align*}
\frac{ [g(z+r) - g(z)] + [g(z-r) - g(z)] }{2} = \frac12 \bE \int_0^\tau g''(z + W_s) \, {\rm d}s.
\end{align*}
\end{remark}

\medskip

Now we specialize the previous formula to the case $\tau_n$ is the exit time of $W^n(t)$ from the ball $B(0,R)$.
We shall see that
an analytic function $g(z)$ is determined by the  average on the boundary of any ball $B(z,R)$. 
This is a generalization of Gauss' mean value theorem for analytic functions, a classical result in complex analysis for the case of Brownian motion.

\begin{theorem}
Let $\tau_n$ be the exit time of $W^n(t)$ from the ball $B(0,R)$.
Then
\begin{equation}
\label{e:173-1}
\lim_{n \to \infty} \tau_n = 0 \qquad \text{a.s.}
\end{equation}
%\begin{equation}
%\label{e:c15}
%g(z) = \lim_{n \to \infty} \bE [g(z + W^n(\tau_n))]
%\end{equation}
\end{theorem}

\begin{proof}
%We claim that
%Then, since $g^{(N)}$ is bounded on $B(z,R)$ for any $z \in \bC$ and $R > 0$, 
%the right-hand side of \eqref{e:183-2} vanishes and the proof is complete.
%\\
%{\color{red} It remains to prove the claim \eqref{e:173-1}.}
%Let us consider the process
%\begin{align*}
%\tilde W^n(s) = W^n(s/n^{1-2/N});
%\end{align*}
%the random walk $\tilde W^n$ has real 2-dimensional variance $\frac{\lfloor t n \rfloor}{n}$ and so, by Donsker's theorem,
%converges to a 2-dimensional Brownian motion as $n \to \infty$ for $N > 2$.
%
%We have
%\begin{align*}
%\tau_n = \inf\{s \,:\, W^n(s) \not\in B(0,R)\} = \inf\{s  \,:\, \tilde W^n(s n^{1-2/N}) \not\in B(0,R)\} = n^{2/N-1} \tilde \tau^n
%\end{align*}
%where  $\tilde \tau_n$ are  the exit times for  $\tilde W^n$ and the $\tilde \tau_n$ converge
%to the hitting time $\tau$ to the boundary of $B(0,R)$ for a 2-dimensional Brownian motion. {\color{red}We claim that
%$\sup_{n} \bE[\tilde \tau_n] < +\infty$.} 
Let $S(t)$, $t \in \bN$, be the random walk defined by the $\xi_j$'s: $\displaystyle S(t) = \sum_{j=1}^t \xi_j$.
It is immediate to see that $S(t)$ is a martingale as well as $|S(t)|^2 - \alpha t$:
\begin{align*}
\bE[|S(t+1)|^2 - |S(t)|^2] = \bE \left[ \sum_{i, j=1}^{t+1} \xi_{i} \bar \xi_{j} - \sum_{i, j=1}^t \xi_{i} \bar\xi_{j} \right]
= \bE \left[|\xi_{t+1}|^2 + \sum_{i=1}^{t}  \xi_{i} \bar \xi_{t+1} + \sum_{j=1}^{t}  \xi_{t+1} \bar \xi_{j}  \right] = \bE[|\xi_{t+1}|^2]
\end{align*}
where in the last equality we used the independence of $\xi_i$'s.
Let $A_n = \{z \in \bC \ \text{s. th.}\ |z| \le R n^{1/N}\}$ and
\begin{align*}
T_{A_n} = \inf \{t \ge 0  \ \text{s. th.}\ S(t) \not\in A_n\}.
\end{align*}
It is $\bP(T_{A_n} < \infty) = 1$ (see e.g.\ [Lawler 2010]); then the stopped martingale
$M^n(t) = |S(t \wedge T_{A_n})|^2  - (t \wedge T_{A_n})$
satisfies
\begin{align*}
0 = \bE[M^n(0)] = \lim_{t \to \infty} \bE \left[ |S(t \wedge T_{A_n})|^2  - (t \wedge T_{A_n}) \right]
= \bE \left[  |S( T_{A_n})|^2  - ( T_{A_n}) \right]
\end{align*}
and using the estimate $n^{2/N} R^2 \le |S(T_{A_n})|^2 \le n^{2/N} R^2 + 1$ it follows that
\begin{align*}
n^{2/N} R^2 \le \bE[T_{A_n}] \le n^{2/N} R^2 + 1.
\end{align*}
Notice that
\begin{align*}
\tau_n = \inf \left\{ s = j/n \ \text{s. th.}\ \frac{1}{n^{1/N}} S(j) \not\in B(0,R) \right\}
\end{align*}
i.e., $\tau_n = \frac{1}{n} T_{A_n}$, so we finally get
\begin{align*}
n^{\frac2N-1} R^2 \le \bE[\tau_n] \le n^{\frac2N-1} R^2 + \frac1n
\end{align*}
which converges to 0 as $n \to \infty$. Since $\tau_n$ is non-negative, this concludes the proof.
\end{proof}

Finally, we have a stochastic characterization of higher order complex derivatives.

\begin{theorem}
Assume that $g$ is an analytic function of exponential type $c$. Define $\tau_n$ be the exit time of $W^n$ from $B(0,R)$.
Then
\begin{align*}
\frac{\alpha}{N!} g^{(N)}(z) = \lim_{n \to \infty} \bE \left[ \frac{1}{\tau_n} [g(z + W^n(\tau_n)) - g(z)] \right].
\end{align*}
\end{theorem}

\begin{proof}
With no loss of generality we set $z=0$.
As in the proof of formula \eqref{e:183-2} we calculate
\begin{align*}
\bE \left[ \frac{1}{\tau_n} [g(W^n(\tau_n)) - g(0)] \right] 
&=
\sum_{k=1}^\infty \frac{1}{k!} \bE \left[ \frac{1}{\tau_n} \int_0^{\tau_n} g^{(k)}(W^n(s)) \, {\rm d}(W^n(s))^k  \right] 
\intertext{since $\tau_n \le T$ we take conditional expectation to get}
&= 
\sum_{k=1}^\infty \frac{1}{k!} \int_0^T \bE \left[ \frac{1}{t} \int_0^{t} g^{(k)}(W^n(s)) \, {\rm d}(W^n(s))^k  \right] \, \bP(\tau_n \in {\rm d}t)
\\
&= 
\sum_{k=1}^\infty \frac{\alpha^k}{(kN)!} \frac{1}{n^{k-1}} \int_0^T  \left[ \frac{1}{t} \int_0^{t} \bE[g^{(kN)}(W^n(s))] \, {\rm d}s  \right] \, \bP(\tau_n \in {\rm d}t).
\end{align*}
From Theorem \ref{t:163} we know that
\begin{align*}
&\bE[g^{(kN)}(W^n(s))] = \gamma_{k}(s) + R_k(n,s), 
\\
&\gamma_k(s) = \sum_{h=0}^\infty \frac{\partial^{(h+k)N}g(0)}{h!} \left(\frac{\alpha s}{N!} \right)^h \quad \text{and $R_k(n,s) \le \frac1n C(\alpha,k,T)$.}
\end{align*}
%{\color{blue!80!red}
By using the bound $|\partial^{(n)}g(0)| \le (c+\varepsilon)^{n}$ for some $\varepsilon > 0$, where $c$ is the exponential type of $g$,
we get also
\begin{align*}
|\gamma_k(s)| \le (c+\varepsilon)^{k N} \exp\left(\frac{\alpha \, s \, (c+\varepsilon)^N}{N!}\right).
\end{align*}

Then we get
\begin{align*}
\bE \left[ \frac{1}{\tau_n} [g(W^n(\tau_n)) - g(0)] \right] 
&=
\sum_{k=1}^\infty \frac{\alpha^k}{(kN)!} \frac{1}{n^{k-1}} \int_0^T  \left[ \frac{1}{t} \int_0^{t} \gamma_{k}(s) + R_k(n,s) \, {\rm d}s  \right] \, \bP(\tau_n \in {\rm d}t)
\\
&=
\sum_{k=1}^\infty \frac{\alpha^k}{(kN)!} \frac{1}{n^{k-1}} \int_0^T  \left[ \frac{1}{t} \int_0^{t} \gamma_{k}(s)  \, {\rm d}s  \right] \, \bP(\tau_n \in {\rm d}t)
+
\sum_{k=1}^\infty \frac{\alpha^k}{(kN)!} \frac{1}{n^{k-1}} \int_0^T  \left[ \frac{1}{t} \int_0^{t} R_k(n,s) \, {\rm d}s  \right] \, \bP(\tau_n \in {\rm d}t)
\end{align*}
By using the integral mean value theorem and the estimate on $R_k$ we get
\begin{align*}
\bE \left[ \frac{1}{\tau_n} [g(W^n(\tau_n)) - g(0)] \right] 
&\approx
\sum_{k=1}^\infty \frac{\alpha^k}{(kN)!} \frac{1}{n^{k-1}} \bE [\gamma_{k}(\varepsilon \tau_n)]
+ \frac1n
\sum_{k=1}^\infty \frac{\alpha^k}{(kN)!} \frac{1}{n^{k-1}} C(\alpha,k,T) %\int_0^T  \left[ \frac{1}{t} \int_0^{t} R_k(n,s) \, {\rm d}s  \right] \, \bP(\tau_n \in {\rm d}t)
\end{align*}
for some $\varepsilon \in (0,1)$.
Recall that $C(\alpha,k,T) \approx c_1 \left( \frac{c_2 k}{\log(1+kN)} \right)^{kN}$, so that the second series in previous formula converges uniformly in $n$;
we write again
\begin{align}\label{e:0502-1}
\bE \left[ \frac{1}{\tau_n} [g(W^n(\tau_n)) - g(0)] \right] 
&\approx
\frac{\alpha}{N!}  \bE [\gamma_{1}(\varepsilon \tau_n)]
+
\frac1n \sum_{k=0}^\infty \frac{\alpha^{k+2}}{((k+2)N)!} \frac{1}{n^{k}} \bE [\gamma_{k+2}(\varepsilon \tau_n)]
+ \frac1n
 C(\alpha,T). %\int_0^T  \left[ \frac{1}{t} \int_0^{t} R_k(n,s) \, {\rm d}s  \right] \, \bP(\tau_n \in {\rm d}t)
\end{align}
%{\color{blue!80!red}
We notice that
\begin{align*}
\sum_{k=0}^\infty \left| \frac{\alpha^{k+2}}{((k+2)N)!} \frac{1}{n^{k}} \bE [\gamma_{k+2}(\varepsilon \tau_n)] \right|
\le
\sum_{k=0}^\infty \frac{\alpha^{k+2}}{((k+2)N)!} (c+\varepsilon)^{(k+2)N} \exp\left(\frac{\alpha \, T \, (c+\varepsilon)^N}{N!}\right)
< + \infty
\end{align*}
hence we can pass to the limit for $n \to \infty$ in \eqref{e:0502-1}, recalling that $\tau_n \to 0$, and we get the thesis.

\end{proof}

\section*{Acknowledgments} 

The financial support of the Fulbright Scholar Program is gratefully acknowledged.
The second author is grateful to the hospitality of the Math Department of the University of Trento.
Thanks also goes to Axel Boldt for useful discussions.

% =========================================================================

%\bibliographystyle{abbrv}
%\bibliography{./../jabref}
%\end{document}
%\section*{Bibliography}

\end{document}